\documentclass[oneside,english]{amsart}
\usepackage[T1]{fontenc}
\usepackage[latin9]{inputenc}
\usepackage{babel}
\usepackage{array}
\usepackage{verbatim}
\usepackage{float}
\usepackage{url}
\usepackage{multirow}
\usepackage{amsthm}
\usepackage{amssymb}
\usepackage{stmaryrd}
\usepackage[all]{xy}
\usepackage[unicode=true,
 bookmarks=true,bookmarksnumbered=false,bookmarksopen=false,
 breaklinks=false,pdfborder={0 0 1},backref=false,colorlinks=false]
 {hyperref}
\hypersetup{pdfauthor={Giosu{\`e} Muratore}}

\makeatletter

\numberwithin{equation}{section}
\numberwithin{figure}{section}
\numberwithin{table}{section}
\theoremstyle{plain} 
\newtheorem{thm}{\protect\theoremname}[section]
\newtheorem{prop}[thm]{\protect\propositionname}

\newtheorem{lem}[thm]{\protect\lemmaname}

\theoremstyle{definition} 
\newtheorem{defn}[thm]{\protect\definitionname}
\newtheorem{example}[thm]{\protect\examplename}

\theoremstyle{remark} 
\newtheorem{rem}[thm]{\protect\remarkname}
\newtheorem*{ack}{\protect\acknowledgmentkname}

\subjclass[2020]{Primary 14N15, 11E81, 14F42}
\usepackage{tikz}
\usetikzlibrary{positioning}

\makeatother

\providecommand{\corollaryname}{Corollary}
\providecommand{\definitionname}{Definition}
\providecommand{\examplename}{Example}
\providecommand{\propositionname}{Proposition}
\providecommand{\remarkname}{Remark}
\providecommand{\theoremname}{Theorem}
\providecommand{\lemmaname}{Lemma}
\providecommand{\problemname}{Problem}
\providecommand{\acknowledgmentkname}{Acknowledgements}
\providecommand{\claimname}{Claim}

\begin{document}
\global\long\def\ev{\mathrm{ev}}%
\global\long\def\GW{\mathrm{GW}}%

\global\long\def\Pnd{\mathcal{P}^n(d)}%
\global\long\def\Ptd{\mathcal{P}^t(d)}%
\global\long\def\P{\mathcal{P}}%
\global\long\def\d{\mathrm{d}}%

\global\long\def\e{\epsilon}%
\global\long\def\V{\mathcal{V}}%

\title{An arithmetic count of osculating lines}
\author{Giosu{\`e} Muratore}
\address{CMAFcIO, Faculdade de Ci\^{e}ncias da ULisboa, Campo Grande 1749-016 Lisboa,
Portugal}
\email{\href{mailto:muratore.g.e@gmail.com}{muratore.g.e@gmail.com}}
\urladdr{\url{https://sites.google.com/view/giosue-muratore}}
\keywords{Osculating curves, bilinear form, Grothendieck--Witt}
\begin{abstract}
    We say that a line in $\mathbb P^{n+1}_k$ is osculating to a hypersurface $Y$ if they meet with contact order $n+1$. When $k=\mathbb C$, it is known that through a fixed point of $Y$, there are exactly $n!$ of such lines. Under some parity condition on $n$ and $\deg(Y)$, we define a quadratically enriched count of these lines over any perfect field $k$. The count takes values in the Grothendieck--Witt ring of quadratic forms over $k$ and depends linearly on $\deg(Y)$.
\end{abstract}

\maketitle

\section{Introduction}
A classical result by Salmon states that, over $k=\mathbb C$, for a general smooth surface $Y\subset \mathbb P^3_{\mathbb C}$ and a general point $p\in Y$ there are exactly two lines meeting $Y$ at $p$ with contact order $3$ \cite{MR0200123}. That result does not depend on the degree of $Y$. A possible generalization of Salmon's result consists on the number of rational curves $C\subset\mathbb P^{n+1}_{\mathbb C}$ meeting a hypersurface with contact order $(n+2)\deg (C)-1$ at a fixed general point. These curves have recently been used in a number of interesting applications \cite{MR3877435}, and their number has been computed recursively using Gromov--Witten invariants in \cite{MR4332489,MR4256011}. They are called osculating curves. In particular, the number of osculating lines to a hypersurface $Y\subset \mathbb P^{n+1}_{\mathbb C}$ is
\begin{equation}\label{eq:n!}
    n!.
\end{equation}

This result was probably first proved in \cite[Proposition~3.4]{MR3773793}. In this paper, we generalize Equation~\eqref{eq:n!} to any perfect field $k$ using $\mathbb A^1$-homotopy. In this theory, the solution of an enumerative problem is an element of the Grothendieck--Witt ring $\GW(k)$ of $k$. It is the completion of the semi-ring of isomorphism classes of nondegenerate, symmetric bilinear forms on finite dimensional vector spaces over $k$. This technique has been used, for example, to compute the number of lines in the cubic surface \cite{MR4247570}, in the quintic threefold \cite{MR4437612}, and the degree of the Grassmannian of lines \cite{MR4237952}.

\subsection{Statement of the main result}
Let $a\in k\setminus\{0\}$. We denote by $\langle a \rangle\in\GW(k)$ the class of the bilinear form $k\times k\rightarrow k$ given by $(x,y)\mapsto axy$. Given a finite extension $K/k$, the group morphism $\mathrm{Tr}_{K/k}\colon\GW(K)\rightarrow \GW(k)$ is the natural map induced by the field trace $K\rightarrow k$. We denote by $\mathbb H:=\langle 1\rangle+\langle -1\rangle$.

Let $Y\subset\mathbb P^{n+1}_k$ be a hypersurface of degree $d\ge n$ given by a homogeneous polynomial $f$, and $p\in Y$ be a point. 
If $p$ is not $k$-rational, we extend the scalars and consider $f_{k(p)}$ with a lift $\tilde p$ of $p$.

If $U$ is an affine open set of $\mathbb{P}^{n+1}_{k(p)}$ centered at $\tilde p$, the restriction of $f_{k(p)}$ to $U$ has a unique Taylor series, that is it decomposes as sum of homogeneous polynomials $f^{(i)}$ of degree $i$:
$${f_{k(p)}|}_{U}=f^{(1)}+f^{(2)}+\cdots+f^{(d)}.$$

If the subscheme $F:= f^{(1)}\cap\cdots\cap f^{(n)}$ of $\mathbb P^n_{k(p)}$ 
is finite, for each point $l\in F$ we denote by $J(l)$ the signed volume of the parallelepiped determined by the
gradient vectors of $f^{(1)},f^{(2)},\cdots,f^{(n)}$ at $l$, and
$$J(Y,p):=\sum_{l\in f^{(1)}\cap\cdots\cap f^{(n)}} 
\mathrm{Tr}_{k(l)/k(p)}\langle J(l)\rangle.$$

Our main result is the following.
\begin{thm}\label{thm:main-1}
    Let $n$ and $d$ be positive integers such that the following conditions are satisfied
    \begin{itemize}
        \item $n\equiv 2 \,\mathrm{mod} \,4$,
        \item $d\equiv 0 \,\mathrm{mod} \,2$.
    \end{itemize}
    Let $Y\subset\mathbb P^{n+1}_k$ be a general hypersurface of degree $d\ge n$, and $\mathcal{L}$ be a general $k$-rational line meeting $Y$ transversely. Then
    $$\sum_{p\in Y\cap\mathcal{L}}\mathrm{Tr}_{k(p)/k}\langle J(Y,p) \rangle=d\frac{n!}{2}\mathbb H.$$
\end{thm}
When $k=\mathbb C$, the map 
$\mathrm{rank}\colon\GW(\mathbb C)\rightarrow \mathbb Z$ is an isomorphism, so the theorem implies Equation~\eqref{eq:n!}. Indeed, each osculating line contributes with $\langle 1\rangle$ to the sum. 
By \cite[Proposition~4.1]{MR4256011}, the number of complex osculating lines is the same for each of the general $d$ points $\{p_1,\ldots,p_d\}=Y\cap\mathcal{L}$. Hence there are $n!$ lines through each of these points.

Our result is related to B{\'e}zout--McKean Theorem, see Remark~\ref{rem:McKean}.

\subsection{Possible generalizations}
One may consider conics instead of lines. For example, the number of osculating conics to a hypersurface $Y\subset \mathbb P^{n+1}_{\mathbb C}$ at $p$ is
$$\frac{(2n+2)!}{2^{n+2}}-\frac{((n+1)!)^2}{2}.$$
This is proved in \cite[Equation~(6.1)]{MR4256011}. See also the tree-formula of \cite{mikhalkin2023ellipsoidal}. Hence, if $n=2$, there are $27$ osculating conics. Darboux \cite{BSMA_1880_2_4_1_348_1} noted that, when $\deg (Y)=3$, each line of $Y$ is coplanar to a unique osculating conic, and vice~versa. This explains why they are $27$. Recent examples of enriched counts of conics are \cite{circles,MR4589636,MR4635342}. We hope to address the problem of the enriched count of osculating curves of higher degree in the near future. We do not expect that the solution will be a multiple of the hyperbolic class $\mathbb H$.
\subsection{Outline}
We introduce $\mathbb A^1$-homotopy in Section~\ref{section:A1}. Sections~\ref{section:Results} and \ref{section:Relative} contain many preliminary results. They follow closely the style of \cite[Section~3]{MR4211099}. Section~\ref{section:Osculating} defines the relavant vector bundle, and the bundle of principal parts. The last section proves the main results. 

\begin{ack}
    The author thanks Ethan Cotterill, Gabriele Degano, Stephen McKean, Kyler Siegel, and Israel Vainsencher for many useful discussions. The author also thank Michael Stillman and Matthias Zach for their computational help, and the Reviewers for taking the necessary time and effort to review the manuscript.  
    This work is supported by FCT - Funda\c{c}\~{a}o para a Ci\^{e}ncia e a Tecnologia, under the project: UIDP/04561/2020 (\url{https://doi.org/10.54499/UIDP/04561/2020}).
    The author is a member of GNSAGA (INdAM).
\end{ack}

\section{Background in \texorpdfstring{$\mathbb A^1$}{A1}-enumerative geometry}
\label{section:A1}
Let $k$ be a field and $X$ be a smooth proper scheme over $k$ of dimension $n$. Moreover, let $E$ be a vector bundle of rank $n$ on $X$.
\begin{defn}\label{defn:rel_ori}
    The bundle $E$ is \emph{relatively oriented} if there exists a line bundle $L$ on $X$, and an isomorphism $\psi\colon\mathrm{Hom}(\det T_X,\det E)\rightarrow L^{\otimes 2}$. The pair $(L,\psi)$ is called a relative orientation of $E$.
\end{defn}
\begin{defn}
    Let $z$ be a point of $X$. A \emph{system of Nisnevich coordinates} around $z$ is a Zariski open neighborhood $U$ of $z$ in $X$, and an {\'e}tale map $\varphi\colon U\rightarrow \mathbb A^n_k$ such that the extension of residue fields $k(\varphi(z))\subseteq k(z)$ is an isomorphism.
\end{defn}
\begin{example}[{\cite[Example~12.1]{MR2242284}}]
    Let $\mathrm{char} (k)\neq 2$ and $a\in k\setminus\{0\}$. The maps $i\colon\mathbb A^1_k\setminus\{a^2\}\hookrightarrow\mathbb A^1_k$ and $(\_)^2\colon\mathbb A^1_k\setminus\{0\}\rightarrow\mathbb A^1_k$ are a system of Nisnevich coordinates around all points of $\mathbb A^1_k$.
\end{example}
\begin{example}
    Let $R:=\mathbb R[x,y]/(x^2+y^2+1)$ 
    and let $p$ be the ideal generated by $y+1$. Let $U$ be the spectrum of the localization $R_p$. The inclusion $\mathbb R[x]\hookrightarrow R_p$ induces a system of Nisnevich coordinates around $z=\mathrm{Spec}(R/p)$.
    
\end{example}
\begin{example}
    If $k\subsetneq K$ is an extension of fields, then the $k$-scheme $\mathrm{Spec}(K)$ does not admit any system of Nisnevich coordinates. Indeed, the residue fields are not isomorphic.
\end{example}
The existence of a system of Nisnevich coordinates 
is granted if $\dim X\ge1$ by 
\cite[Proposition~20]{MR4247570}. If $E$ is a vector bundle and $\varphi\colon U\rightarrow\mathbb A^n_k$ is a system of Nisnevich coordinates given by $\varphi=(\varphi_1,\ldots,\varphi_n)$, by shrinking $U$ we may suppose without loss of generality that there exists a trivialization $\widehat{\varphi}\colon{E|}_{U}\rightarrow \mathcal{O}_U^{\oplus n}$. This local trivialization induces a distinguished element $\det(\widehat{\varphi}):=\varphi_1\wedge\ldots\wedge\varphi_n$ of $\Gamma(U,\det E)$ . On the other hand, there is a distinguished element of $\Gamma(U,\det T_X)$ induced by the standard basis of $T_{\mathbb A^n_k}$.

\begin{defn}\label{defn:compatible}
    Let $(L,\psi)$ be a relative orientation of $E$. Let $\varphi\colon U\rightarrow \mathbb A^n_k$ be a system of Nisnevich coordinates, and let $\widehat{\varphi}\colon{E|}_{U}\rightarrow \mathcal{O}_U^{\oplus n}$ be a trivialization. Let $\lambda$ be the linear map sending the distinguished element of $\Gamma(U,\det T_X)$ to the distinguished element $\det(\widehat{\varphi})$ of $\Gamma(U,\det E)$. We say that $\psi$ is \emph{compatible} with $\varphi$ and the trivialization $\widehat{\varphi}$ if the induced map on global sections
    \begin{equation*}
        {\psi|}_U\colon \mathrm{Hom}_k(\Gamma(U,\det T_X),\Gamma(U,\det E))\longrightarrow \Gamma(U,L^{\otimes 2})
    \end{equation*}
    sends $\lambda$ to a square element of $\Gamma(U,L)^{\otimes 2}$.
\end{defn}

In differential topology, the Brouwer degree of a differentiable map $f\colon X\rightarrow Y$ can be computed in the following way. Take $y\in Y$ a regular value, so that the Jacobian matrix $J_f$ is invertible for each one of the finite points $x\in X$ such that $f(x)=y$. The local Brouwer degree at $x$ is $+1$ if $\det(J_f)|_x>0$, and $-1$ otherwise. The Brouwer degree is the sum of all local degrees, that is
\begin{equation}\label{eq:brouwer}
    \deg(f)=\sum_{x\in f^{-1}(y)} \mathrm{sgn}(\det(J_f)|_x),
\end{equation}
and it depends only on the homotopy class of $f$. See \cite[Chapter 17]{MR2954043} for a nice introduction.

Morel extended this notion by defining the $\mathbb A^1_k$-homotopy degree. Roughly speaking, in \cite{MR1813224} a category is constructed upon algebraic varieties over $k$, but allowing many tools of usual topology of manifolds. In particular, the unit interval is replaced by $\mathbb A^1_k$ in homotopy. In this new category 
there is a Brouwer degree 
taking value in 
$\GW(k)$ 
(see \cite[Corollary 1.24]{MR2934577} and reference therein, and \cite{MR2275634} for an introduction). Morel's construction has been made explicit in various contexts (see \cite{MR3059240,MR4648854}).

\subsection{Euler class of a vector bundle}
Let $(L,\psi)$ be a relative orientation of $E$ as in Definition~\ref{defn:rel_ori}. The map $\psi$ induces an isomorphism $\psi'\colon\det E^\vee\otimes L^{\otimes 2} \rightarrow \omega_X$. For each $0\le a,b\le n$, we may use $\psi'$, the graded Koszul complex \cite[17.2]{MR1322960}, and Serre duality to construct a perfect pairing
\begin{equation*}
    \beta_{a,b}\colon H^a(X,\wedge^b E^\vee\otimes L)\otimes H^{n-a}(X,\wedge^{n-b} E^\vee\otimes L)\longrightarrow k.
\end{equation*}
When $2a=2b=n$, $\beta_{a,b}$ is a bilinear form. For all other values of $a$ and $b$, the direct sum $\beta_{a,b}\oplus\beta_{n-a,n-b}$ is also a nondegenerate symmetric bilinear form.

\begin{defn}\label{defn:GSd_Euler}
    The Grothendieck--Serre--duality \emph{Euler number} of $E$ is $$e(E):=\sum_{0\le a,b\le n}(-1)^{a+b}\beta_{a,b}\in\GW(k).$$
\end{defn}
See \cite[Section 2]{MR4557905}. This definition does not depend on the relative orientation.

An older but more explicit definition of Euler class was given in \cite{MR4247570}. It requires a section $\sigma\colon X\rightarrow E$ with isolated zeros.




Let $z$ be a zero of $\sigma$, and let 
$\deg^{\mathbb A_k^1}_z(\sigma)\in\GW(k)$ 
denote its local $\mathbb A^1_k$-degree. In the case where $z$ is simple, has a compatible Nisnevich neighborhood isomorphic to the affine space, and $k(z)/k$ is separable, then by \cite[Proposition 15]{MR3909901} we have
\begin{equation}\label{eq:ind}
   \deg^{\mathbb A_k^1}_z(\sigma) = \mathrm{Tr}_{k(z)/k} \langle\det(J_{\sigma})|_{z}\rangle.
\end{equation}
The Jacobian $J_{\sigma}$ of $\sigma$ is computed locally at $\sigma\colon\mathbb A_k^n\rightarrow\mathbb A_k^n$. 

More generally, if $z$ is not simple,  $\deg^{\mathbb A^1_k}_z(\sigma)$ is constructed in a different way. Let 
$Z$ denote the zero locus of $\sigma$. There is a natural isomorphism
$$\mathrm{Hom}_k(\mathcal{O}_{Z,z},k)\cong\mathcal{O}_{Z,z}$$
of $\mathcal{O}_{Z,z}$-algebras. The Scheja--Storch form (\cite{MR0393056}) is the map $\eta$ corresponding to $1$ under this isomorphism. This defines a bilinear form $x\otimes y\mapsto \eta(xy)$ on $\mathcal{O}_{Z,z}$ (see \cite{MR0458467,MR0467800,MR0494226}), and $\deg^{\mathbb A_k^1}_z(\sigma)$ is the class of this bilinear form as proved in \cite[Main~Theorem]{MR3909901}. 
Finally
\begin{equation}\label{eq:euler}
    e(E)=\sum_{z\in\sigma^{-1}(0)}\deg^{\mathbb A^1_k}_z(\sigma),
\end{equation}
by \cite[Theorem 1.1]{MR4557905}. Equation~\eqref{eq:ind} makes clear the analogy between \eqref{eq:euler} and the classical Brouwer degree \eqref{eq:brouwer}.



\subsection{The non-orientable case}
Many interesting vector bundles are not relatively orientable, 
so that there is no Euler class as in Definition~\ref{defn:GSd_Euler}. Larson and Vogt partially solve this problem by introducing the relative orientability relative to a divisor.
\begin{defn}
    Let $X$ be a smooth variety, $E\rightarrow X$ be a vector bundle and $D\subset X$ be an effective divisor. We say that $E$ is \emph{relatively orientable relative to} $D$ if there exists a line bundle $L$ on $X$, and an isomorphism $$\psi\colon\mathrm{Hom}(\det T_X,\det E)\otimes\mathcal{O}_X(D) \longrightarrow L^{\otimes 2}.$$
    Equivalently, $E$ is relatively orientable on the open subvariety $X\setminus D$.
\end{defn}
At least over $\mathbb R$, this definition allows us to define the Euler class of a vector bundle that does not satisfy Definition \ref{defn:rel_ori}. So, let $X$ be a smooth real variety. Let $\sigma$ be a section of $E$ with isolated zeros. For each zero $z$, the class of the Scheja--Storch form $\deg^{\mathbb A_{\mathbb R}^1}_z(\sigma)$ is still well-defined. Let us denote by $V_D\subset H^0(E)$ the locus of $\mathbb R$-points of $H^0(E)$ with a real zero along $D$. 
\begin{lem}\label{lem:LV21}
    Let $X$ be a smooth real projective variety and $E$ be a vector bundle relatively oriented relative to an effective divisor $D$. Let $H^0(E)^\circ$ denote the space of sections with isolated zeros. Then
    $$\sum_{z\in\sigma^{-1}(0)}\deg^{\mathbb A^1_{\mathbb R}}_z(\sigma)$$
    is constant for $\sigma$ in any real connected component of $H^0(E)^\circ\setminus V_D$.
\end{lem}
\begin{proof}
    See \cite[Lemma 2.4]{MR4253146}.
\end{proof}
Suppose that $1_D$ is a global section of $\mathcal{O}_X(D)$ that gives $D$. For a vector bundle $E$ relatively orientable relative to $D$, we say that the relative orientation $(L,\psi)$ is compatible with a system of Nisnevich coordinates $\varphi\colon U\rightarrow \mathbb A^n_k$ if $\lambda\otimes 1_D$ is a square, where $\lambda$ is like in Definition~\ref{defn:compatible}. See \cite[Definition 3.7]{MR4211099} for more details.

\section{Results on the flag variety}\label{section:Results}
Let $n$ be a nonnegative integer, and $k$ be a perfect field. Let $i\in\{0,\ldots,n+1\}$, we denote by $U_i$ the open subset
$$U_i:=\mathrm{Spec}\left ( k\left[\frac{x_0}{x_i},\ldots,\frac{x_{n+1}}{x_i}\right]\right ) \subset \mathrm{Proj}\left ( k[x_0,\ldots,x_{n+1}]\right ) =:\mathbb{P}^{n+1}_k.$$
\begin{rem}
    In \cite[Section 2]{MR4211099}, $U_i$ is defined as the open subset of $\mathbb{P}^{n+1}_k$ parameterizing points $[p_0:\ldots: p_{n+1}]$ such that $p_i\neq 0$. This notation makes sense only for $k$-points, as it denotes the point corresponding to the homogeneous ideal $$(p_ix_0-p_0x_i,\ldots, p_{i}x_{n+1}-p_{n+1}x_i).$$
    Nevertheless, we adopt the same convention. In particular, we characterize any morphism by its action on the \emph{coordinates} $\{x_0,\ldots,x_{n+1}\}$.
\end{rem}

Let $G=G(2,n+2)$ be the Grassmannian of lines in $\mathbb{P}^{n+1}_k$. It is the space of matrices of order $2\times (n+2)$ of maximal rank, modulo the action of $\mathrm{GL}(2,k)$. In particular, for each pair of distinct indices $\{i,\alpha\}\subset\{0,\ldots,n+1\}$, there is an open subset isomorphic to $\mathbb A^{2n}_k$ parameterizing those orbits in $G$ whose submatrix containing the columns $i$ and $\alpha$ is invertible. See, for example, \cite[8.1]{MR2004218}. We denote by $\mathcal{O}_G(-1)$ the determinant of the tautological vector bundle $\mathcal{S}$ of $G$, see p.~248 of {\it loc. cit.}. We denote by $X$ the variety of flags
$$\{0\}\subset k\subset k^2\subset k^{n+2}.$$
It has two natural maps:
\begin{equation*}
    \pi\colon X\rightarrow G, \,\,\,\,\, \ev\colon X\rightarrow \mathbb P^{n+1}.
\end{equation*}
The pair 
$(X,\ev)$ is naturally isomorphic to the projectivization of tangent bundle of $\mathbb P^{n+1}_k$, thus it is a $\mathbb P^n_k$-bundle over $\mathbb P^{n+1}_k$. We denote by $\mathcal{O}_X(1)$ the line bundle $\ev^*\mathcal{O}_{\mathbb P^{n+1}_k}(1)$. We will use the following properties of the Grassmannian.
\begin{prop}
    Let $p\in\mathbb P^{n+1}_k(k)$ and $Z=\ev^{-1}(p)$. In the notation above, we have the following:
    \begin{align}
    \omega_G &\cong \mathcal{O}_G(-n-2) \label{eq:omega_G}, \\
    \pi^*\mathcal{O}_G(1) &\cong \omega_\pi\otimes\mathcal{O}_{X}(2), \label{eq:pi_star_G} \\
    {\omega_{\pi}|}_Z &\cong \mathcal{O}_{\mathbb P^{n}_{k}}(1). \label{eq:omega_Z}
\end{align}
\end{prop}
\begin{proof}
    The first equation follows by \cite[Proposition~5.25]{MR3617981}. In order to prove the second formula, note that $(X,\pi)$ is naturally isomorphic the projectivization of the vector bundle $\mathcal{S}$. Theorem~11.4 of {\it loc. cit.} implies that 
    $$c_1(\omega_\pi^\vee)=\binom{1+1-0}{1-0}c_1(\mathcal{O}_X(1))+\binom{1+1-1}{1-1}c_1(\pi^*\mathcal{S}).$$
    Since $c_1(\pi^*\mathcal{S})=\pi^*c_1(\det\mathcal{S})=\pi^*c_1(\mathcal{O}_G(-1))$, we easily deduce Equation~\eqref{eq:pi_star_G}. 
    
    In order to prove the last one, 
    note that $Z$ parameterizes the set of flags
    $$\{0\}\subset p\subset k^2\subset k^{n+2}.$$
    Dividing the above equation by $p$, we see that $Z$ is naturally isomorphic to the variety of projective lines of $k^{n+2}/p$, which is $\mathbb P^{n}_{k}$. Thus there is a natural embedding $i\colon Z\hookrightarrow G$. This embedding preserves the Schubert cycles as discussed in Chapter~4.1 of {\it loc. cit.}, thus $i^*\mathcal{O}_G(1)=\mathcal{O}_Z(1)$ and the result follows from the second equation. 
\end{proof}
\begin{rem}
    If $p$ is not $k$-rational, then  $Z=\mathbb P(T_{\mathbb P^{n}_k,p}) \cong\mathbb P^{n}_{k(p)}$. Since $G$ is covered by affine spaces, the equation ${\omega_{\pi}|}_Z = \mathcal{O}_{Z}(1)$ could still be proved using Galois descent (see \cite[Example~6.2/B]{MR1045822}).
\end{rem}
Given a set of indeces $I$ and a sequence of elements $\{a_i\}_{i\in I}$, by $(\ldots,\widehat{a}_{j},\ldots)$ we mean that $a_j$ is missing from the sequence. 

There is a standard system of Nisnevich coordinates of $\mathbb P^{n}_k$ that we denote by $\{(U_i,\varphi_i)\}_{i=0}^n$. That is
\begin{equation}
    \varphi_i([p_0:\ldots: p_{n+1}])=\left (\frac{p_0}{p_i},\ldots,\frac{p_{i-1}}{p_i},\frac{p_{i+1}}{p_i},\ldots,\frac{p_n}{p_i}\right).
\end{equation}
These coordinates give a natural system of coordinates of $T_{\mathbb P^{n+1}_k}$ in the following way. Let $(x_{0},\ldots,\widehat{x}_{j},\ldots,x_{n+1})\in\mathbb A^{n+1}_k$ be local coordinates. The maps 
\begin{equation}\label{eq:trans_P}
    \varphi_i\circ\varphi_j^{-1}(x_{0},\ldots,\widehat{x}_{j},\ldots,x_{n+1})=\left(\ldots,\frac{x_{j-1}}{x_i},\frac{1}{x_i},\frac{x_{j+1}}{x_i},\ldots,\frac{x_{n+1}}{x_i}\right)\in\mathbb A^{n+1}_k
\end{equation}
are the standard gluing of the projective space along $U_j\cap U_i$. In order to obtain the standard gluing of the tangent bundle with respect to the standard basis $\{\frac{\partial}{\partial{x_t}}\}_{t\neq j}$, we take the Jacobian $J_{ij}$ of \eqref{eq:trans_P}, leading to maps
\begin{equation}\label{eq:trans_J}
    J_{ij}\left(\frac{\partial}{\partial{x_s}}\right)
    =
    \begin{cases}
        \frac{1}{x_i}\frac{\partial}{\partial{x_s}} & \text{if $s\neq i,j$}\\
        -\frac{1}{x_i^2}\frac{\partial}{\partial{x_i}}-\underset{t\neq i,j}{\sum}\frac{x_t}{x_i^2}\frac{\partial}{\partial{x_t}} & \text{if $s=i$}.
    \end{cases}
\end{equation}
Thus the cocycle $\det(J_{ij})\circ \varphi_j$ maps the distinguished basis
$$\frac{\partial}{\partial{x_0}}\wedge\ldots\wedge\widehat{\frac{\partial}{\partial{x_{j}}}}\wedge\ldots\wedge \frac{\partial} {\partial{x_{n+1}}}$$
to
$$J_{ij}\left(\frac{\partial}{\partial{x_0}}\right)\wedge\ldots\wedge J_{ij}\left(\frac{\partial}{\partial{x_{j-1}}} \right)\wedge J_{ij}\left(\frac{\partial}{\partial{x_{j+1}}} \right)\wedge\ldots\wedge J_{ij}\left(\frac{\partial} {\partial{x_{n+1}}}\right),$$
which is equal to
$$\frac{(-1)^{i+j}}{x_i^{n+2}}  \frac{\partial}{\partial{x_0}}\wedge\ldots\wedge\widehat{\frac{\partial}{\partial{x_{j}}}}\wedge\ldots\wedge\frac{\partial}{\partial{x_{n+1}}}.$$
The Jacobian determinant is $-1/x_i^{n+2}$, but we multiply by $(-1)^{i+j+1}$ after reordering the derivations. See the proof of \cite[Proposition~2.4.3]{MR2093043}. We will use another system of coordinates, which is better suited for our problem. 
\begin{defn}
    The twisted open cover $\{(U_i,\tilde\varphi_i)\}_{i=0}^n$ of $\mathbb P^{n}_k$ is defined as $\tilde\varphi_0=\varphi_0$ and 
    \begin{equation}\label{eq:mckean_coo}
        \tilde\varphi_i([p_0:\ldots: p_{n}])=\left((-1)^i\frac{p_0}{p_i},\ldots,\frac{p_{i-1}}{p_i},\frac{p_{i+1}}{p_i},\ldots,\frac{p_n}{p_i}\right).
\end{equation}
\end{defn}
\begin{prop}\label{prop:distinguished}
    The twisted covering maps $\{(U_i,\tilde\varphi_i)\}_{i=0}^{n+1}$ are a system of Nisnevich coordinates. Moreover, for each pair of indeces $(i,j)$, the morphism $\varphi_i$ determines: the distinguished basis element
    $$(-1)^i\cdot \partial_i:=(-1)^i\bigwedge_{t\neq i}\frac{\partial}{\partial(x_t/x_i)}\in \Gamma\left(U_i,\det {T_{\mathbb P^{n+1}_k}|}_{U_i}\right),$$
    and the transition functions $g_{ij}\colon {T_{\mathbb P^{n+1}_k}|}_{U_j}\rightarrow {T_{\mathbb P^{n+1}_k}|}_{U_i}$ such that
    \begin{equation}\label{eq:detg}
        \det g_{ij}=(-1)^{i+j}\left(\frac{x_i}{x_j}\right)^{n+2}.
    \end{equation}
\end{prop}
\begin{proof}
    See \cite[Proposition~3.8]{MR4211099}.
\end{proof}
We will use the coordinates described in Equation~\eqref{eq:mckean_coo} to a natural atlas of $X$.
\begin{defn}\label{defn:my_Nis}
    Let $i,\alpha\in\{0,\ldots,n+1\}$ be such that $\alpha\neq i$. Let $\{\frac{\partial}{\partial x_t}\}_{t\neq i}$ be a trivializing base of $T_{U_i}$. Let $U_\alpha\subset \mathbb P(T_{U_i})$ be the open set corresponding to $\frac{\partial}{\partial x_\alpha}\neq 0$. Thus $\ev^{-1}(U_i)=U_i\times \mathbb P^{n}_k$ is covered by open subsets $U_{i\alpha}:=U_i\times U_\alpha$. The morphism $\tilde\varphi_{i\alpha}\colon U_{i\alpha}\rightarrow\mathbb A^{n+1}_k\times\mathbb A^{n}_k$ is defined as follows
\begin{equation}
    \tilde\varphi_{i\alpha}([p_0:\ldots: p_{n+1}]\times[q_0:\ldots:q_n]):= \varphi_{i}([p_0:\ldots: p_{n+1}])\times \tilde\varphi_{\alpha}([q_0:\ldots:q_n]).
\end{equation}
\end{defn}
Since $\tilde\varphi_{i\alpha}$ is the direct product of two isomorphisms, it is an isomorphism. In particular it gives a system of Nisnevich coordinates. 

\begin{prop}\label{prop:prop3.8}
    The covering maps $\{(U_{i\alpha},\tilde\varphi_{i\alpha})\}$ are Nisnevich coordinates. 
    
    Moreover, $\tilde\varphi_{i\alpha}$ determines a distinguished basis element of $\det {T_{X}|}_{U_i}$ with transition functions $g_{(i\alpha),(j\beta)}\colon {{T_X}|}_{U_{j\beta}}\rightarrow {{T_X}|}_{U_{i\alpha}}$ such that
\begin{equation}\label{eq:det_g}
    \det g_{(i\alpha),(j\beta)}=\left(\frac{x_i}{x_j}\right)^{2}\left(\frac{x_iy_\alpha}{x_jy_\beta}\right)^{n+1}.
\end{equation}
\end{prop}
\begin{proof}
    In order to obtain the gluing functions of $X$, we need to combine \eqref{eq:trans_P} and \eqref{eq:trans_J}. A few calculations 
    lead to the following maps of $\mathbb A^{n+1}_k\times \mathbb A^{n}_k$:
    \begin{align}\label{eq:varphi}
        \tilde\varphi_{i\alpha}\circ\tilde\varphi_{j\beta}^{-1}
        ((x_{0},\ldots,\widehat{x}_{j},\ldots,x_{n+1}) &\times
        (y_{0},\ldots,\widehat{y}_{j},\ldots,\widehat{y}_{\beta},
        \ldots,y_{n+1})) \\ \nonumber
        =
        (x_{0}',\ldots,\widehat{x}_{i}',\ldots,x_{n+1}') &\times 
        ((-1)^{\alpha+\beta}
        y_{0}',\ldots,\widehat{y}_{i}',\ldots,\widehat{y}_{\alpha}',
        \ldots,y_{n+1}'),
    \end{align}
    where for $s\neq i,j,\beta$,
    $$x_{s}' =\frac{x_s}{x_i}, \,\,\,\,\,
        x_{\beta}' =\frac{x_\beta}{x_i}, \,\,\,\,\,
        x_{j}' =\frac{1}{x_i},$$
    $$y_{s}' =\frac{y_s}{y_\alpha}, \,\,\,\,\,
        y_{\beta}' =\frac{1}{y_\alpha}, \,\,\,\,\,
        y_{j}' =\frac{-1}{x_i}\frac{y_i}{y_{\alpha}}-\frac{x_\beta}{x_i}\frac{1}{y_{\alpha}}-\sum_{t\neq i,j,\beta}\frac{x_t}{x_i}\frac{y_{t}}{y_{\alpha}}.$$
    The determinant of the Jacobian matrix of \eqref{eq:varphi} is
    $$J_{ij\alpha\beta}:=(-1)^{\alpha+\beta}\det J_{ij}\det\left(\frac{\partial y_t'}{\partial y_s}\right).$$
    Note that in order to compute it, we can take $y_j'=-y_i/(x_iy_\alpha)$ since the linear combination of the other coordinates $y_s'$ does not influence the determinant. Thus
    \begin{align*}
        \det J_{ij}\det\left(\frac{\partial y_t'}{\partial y_s}\right) &= (-1)^{\alpha+\beta}\left(\frac{1}{x_i}\right)^{n+2}\left(\frac{-1}{x_i}\right)(-1)\left(\frac{1}{y_\alpha}\right)^{n+1} \\
         &= (-1)^{\alpha+\beta}\left(\frac{1}{x_i}\right)^2\left(\frac{1}{x_iy_\alpha}\right)^{n+1}.
    \end{align*}
    We have a distinguished basis element of $\det {{T_X}|}_{U_{i\alpha}}$ given by
    $$(-1)^{\alpha}\cdot \partial_i\wedge\partial_\alpha :=
    \bigwedge_{t\neq i}\frac{\partial}{\partial(x_t/x_i)}\wedge
    (-1)^\alpha\bigwedge_{t\neq i,\alpha}\frac{\partial}{\partial(y_t/y_\alpha)}.$$
    After reordering, we see that the cocycles $J_{ij\alpha\beta}\circ\tilde\varphi_{j\beta}$ equal \eqref{eq:det_g} and are compatible with the distinguished basis.
    The reason for the compatibility is that the transition map in \eqref{eq:varphi} acts with a permutation $(j+1,i)$ if $j<i$ (or $(i+1,j)$, if $j>i)$ on the coordinates $\{x_s\}_{s\neq j}$. On the other hand, if $j<i$ and $\beta<\alpha$ it acts with a permutation $(j+1,i)(\beta+1,\alpha)$ on the coordinates $\{y_s\}_{s\neq j,\beta}$. In any case, the signature of \eqref{eq:varphi} is $(-1)^{\alpha+\beta+1}$.
\end{proof}
\begin{lem}\label{lem:local_triv}
    The line bundles $\mathcal{O}_X(1)$ and $\pi^*\mathcal{O}_G(1)$ are locally trivialized over $U_{i\alpha}$, respectively, by
    $$\left(\frac{x_i}{x_0}\right),  \,\,\,\,\, \left(\frac{x_iy_\alpha}{x_0y_1-x_1y_0}\right).$$
    Moreover, the transition functions are, respectively,
    $$\left(\frac{x_i}{x_j}\right),  \,\,\,\,\,\left(\frac{x_iy_\alpha}{x_jy_\beta}\right).$$
\end{lem}
\begin{proof}
    On $\ev^{-1}(U_i)$ (thus, on $U_{i\alpha}$), $\mathcal{O}_X(1)$ is trivialized by the pullback of $(\frac{x_i}{x_0})$, which is the trivializer of $\mathcal{O}_{\mathbb P^{n+1}_k}(1)$ on $U_i$. The induced transition functions on $\mathcal{O}_X(1)$ are those of Equation~\eqref{eq:detg}; see \cite[Proposition~3.9]{MR4211099}.
    In order to trivialize $\pi^*\mathcal{O}_G(1)$, we recall that $G$ is defined as the space of $\mathrm{GL}(2,k)$-orbits of $2\times(n+2)$ matrices. In particular, we may see $X$ as the space of pairs
    \begin{equation*}
     [x_0:\ldots:x_{n+1}]\times\left[
     \begin{pmatrix}
         x_0 & \ldots & x_{n+1} \\
         y_0 & \ldots & y_{n+1}
     \end{pmatrix}
     \right]_{\mathrm{GL}(2,k)}.
    \end{equation*}
    Suppose that $[x_0:\ldots:x_{n+1}]\in U_i$. Using the fact that $x_i\neq 0$ all orbits appearing in the above equation admit a representative with $y_i=0$. In particular, we may think that $G$ is the moduli space of orbits of the form 
    \begin{equation}\label{eq:matrix_form}
        \begin{pmatrix}
            x_0 & \ldots & x_i & \ldots & x_{n+1} \\
            y_0 & \ldots & 0   & \ldots & y_{n+1}
        \end{pmatrix}.
    \end{equation}
    Moreover, the coordinates of $\ev^{-1}(U_i)$ are naturally identified with the projective points $[x_0:\ldots:x_{n+1}]$ and $[y_0:\ldots:0:\ldots:y_{n+1}]$ appearing above.
    
    Let $\Pi\colon G\rightarrow \mathbb P(\wedge^2 k^{n+2})$ be the Pl{\"u}cker embedding. That is, the embedding sending a matrix like in Equation~\eqref{eq:matrix_form} to the point $[\ldots:P_{ts}:\ldots]$ where $t<s$ and $P_{ts}=x_ty_s-x_sy_t$. Suppose that $i<\alpha$. Note that the composition $\Pi\circ\pi$ defines an isomorphism between $U_{i\alpha}$ and $\pi(G)\cap \{P_{i\alpha}\neq 0\}$. 

    Since $\pi$ is surjective and $\Pi$ is not degenerate, the map $U_{i\alpha}\rightarrow \mathbb A^1_k$ defined by $(P_{01}/P_{i\alpha})$ is regular and nonzero. 
    
    As above, 
    $\mathcal{O}_{\mathbb P(\wedge^2 k^{n+2})}(1)$ is trivialized over $\{P_{i\alpha}\neq 0\}$ by $(P_{i\alpha}/P_{01})$. Thus by pullback we get that 
    $\pi^*\mathcal{O}_G(1)=\pi^* \Pi^*\mathcal{O}_{\mathbb P(\wedge^2 k^{n+2})}(1)$ is trivialized by $x_iy_\alpha/(x_0y_1-x_1y_0)$ with transition functions $(x_iy_\alpha/x_jy_\beta)$. If $\alpha<i$, we take without loss of generality the map $-(P_{01}/P_{\alpha i})$ and repeat the same argument.
\end{proof}

\section{Osculating lines}\label{section:Osculating}
Let $Y\subset \mathbb P^{n+1}_k$ be a smooth hypersurface, and $\mathcal{L}\subset \mathbb P^{n+1}_k$ be a line. We say that $\mathcal{L}$ and $Y$ have $t$-contact with each other if $\mathcal{L}\cap Y$ contains a divisor of $\mathcal{L}$ of the form $tp$ for some point $p\in \mathcal{L}$.
 
We say that $\mathcal{L}$ is osculating to $Y$ if they are $(n+1)$-contact. 

Lines with $(t+1)$-contact with $Y$ are parameterized by the zero locus of a section of the $t^\mathrm{th}$ bundle of principal parts, see \cite[Section 5]{MR0626480}.
\begin{defn}
    Let $d$ be an integer. For every $t\ge0$, the $t^\mathrm{th}$ \emph{bundle of principal parts} of $\mathcal{O}_X(d)$ is
    \begin{equation}
        \Ptd:=q_*(p^*\mathcal{O}_{X}(d)\otimes \mathcal{O}_{X\times_G X}/{\mathcal{I}_\Delta^{t+1}}),
    \end{equation}
    where $p$ and $q$ are the two projections $X\times_G X\rightarrow X$.
\end{defn}
There exists a canonical differential morphism of sheaves $\partial^t\colon \mathcal{O}_X(d) \rightarrow \Ptd$ (see \cite[D{\'e}finition (16.3.6)]{MR238860}). For any section $s$ of $\mathcal{O}_X(d)$, the section $\partial^t s$ of $\Ptd$ is called principal part of $s$. Geometrically, $\partial^t$ maps a local section defined at a point $x\in X$ to its first $t$ Hasse--Schmidt derivatives relative to $\pi$. This is because $(\partial^t,\Ptd)$ represents the functor sending a line bundle $M$ to the set of $\mathcal{O}_G$-linear $t$-derivations $D\colon \mathcal{O}_X(d)\rightarrow M$ relative to $\pi$ \cite[Proposition (16.8.4)]{MR238860}. See \cite{MR2349665,MR4626882} for a more explicit approach.
Finally the exact sequence
$$
0\rightarrow {\mathcal{I}^{t+1}_\Delta}/{\mathcal{I}^{t+2}_\Delta}
\rightarrow \mathcal{O}_{X\times_G X}/\mathcal{I}_\Delta^{t+2} \rightarrow \mathcal{O}_{X\times_G X}/\mathcal{I}_\Delta^{t+1}\rightarrow 0
$$
induces the exact sequence
\begin{equation}\label{eq:principal}
    0\rightarrow \omega_{\pi}^{\otimes t+1}\otimes \mathcal{O}_X(d)
\rightarrow \P^{t+1}(d) \rightarrow \P^{t}(d)\rightarrow 0.
\end{equation}

Finally, we define the following vector bundle.

\begin{defn}
    We denote by $\V$ the vector bundle on $X$ given by $\Pnd\oplus \mathcal{O}_X(1)^{\oplus n}$.
\end{defn}

\section{Relative orientation}\label{section:Relative}
The rank of $\V$ is clearly $2n+1$. Moreover, using the property of the determinant of a vector bundle, it is easy to see that
\begin{equation}\label{eq:det_V}
    \det\V=\mathcal{O}_X(1)^{\otimes n}\otimes\mathcal{O}_X(d)^{\otimes n+1}\otimes\omega_\pi^{n(n+1)/2}.
\end{equation}
Indeed, $\det\V=\det(\Pnd)\otimes\mathcal{O}_{X}(n)$. Using \eqref{eq:principal} and an easy induction,
\begin{align*}
    \det(\Pnd)  & =  \det(\P^{n-1}(d))
    \otimes\omega_\pi^{\otimes n}\otimes\mathcal{O}_{X}(d)\\
                & = \omega_\pi^{\otimes n(n+1)/2}\otimes\mathcal{O}_{X}((n+1)d).
\end{align*}
We use this equation for trivializing $\det\V$.
\begin{prop}\label{prop:prop3.9}
    The line bundle $\det\V$ is trivialized on $U_{i\alpha}$ by the distinguished element
    \begin{equation}\label{eq:prop3.9-1}
        \left(\frac{x_i}{x_0}\right)^{n+d(n+1)}\left(\frac{x_i}{x_0}\right)^{-n(n+1)}\left(\frac{x_iy_\alpha}{x_0y_1-x_1y_0}\right)^{n(n+1)/2},
    \end{equation}
    with transition functions 
    $\det{{\V}|}_{U_{j\beta}}\rightarrow\det{{\V}|}_{U_{i\alpha}}$ 
    being
    \begin{equation}\label{eq:prop3.9-2}
        \left(\frac{x_i}{x_j}\right)^{d(n+1)-n^2}\left(\frac{x_iy_\alpha}{x_jy_\beta}\right)^{n(n+1)/2}.
    \end{equation}
\end{prop}
\begin{proof}
    By Equation~\eqref{eq:det_V}, we need only the local trivializations of $\omega_\pi$ and $\mathcal{O}_X(1)$. They are easily deduced from Lemma~\ref{lem:local_triv} and Equation~\eqref{eq:pi_star_G}. In particular, ${\omega_{\pi}|}_{U_{i\alpha}}$ is trivialized by 
    $$\left(\frac{x_i}{x_0}\right)^{-2}\left(\frac{x_iy_\alpha}{x_0y_1-x_1y_0}\right).$$
    Finally,  Equations~\eqref{eq:prop3.9-1} and \eqref{eq:prop3.9-2} follow from \eqref{eq:det_V}.
\end{proof}

\subsection{Relatively orientable case}
The following proposition gives us the conditions on $n$ and $d$ such that $\V$ is relatively orientable.
\begin{prop}\label{prop:conditionsnd}
    The vector bundle $\V$ is relatively orientable if and only if the following conditions are satisfied:
    \begin{itemize}
        \item $n\equiv 2 \,\mathrm{mod} \,4$,
        \item $d\equiv 0 \,\mathrm{mod} \,2$.
    \end{itemize}
\end{prop}
We prove this proposition after the following two simple lemmas.
\begin{lem}\label{lem:cong23}
    The following conditions are equivalent:
    \begin{enumerate}
        \item $n\equiv 2 \,\mathrm{mod} \,4$, or $n\equiv 3 \,\mathrm{mod} \,4$.
        \item $\sum_{t=0}^n t \equiv (n+1) \,\mathrm{mod} \,2$.
    \end{enumerate}
\end{lem}
\begin{proof}
    Since the sum of four consecutive integers is even, we may take the remainder of $n$ modulo $4$. By a direct computation, we see that $n\in\{2,3\}$ satisfies the condition, but $n\in\{0,1\}$ does not.
\end{proof}
\begin{lem}
    Let $n\ge0$ and $d>0$ be integers. Then
    \begin{equation}\label{eq:Hom}
        \mathrm{Hom}(\det(T_{X}),\det(\V)) = \omega_{\pi}^{\otimes (n+1)(n/2-1)}
        \otimes\mathcal{O}_{X}((n+1)(d-1)-3).
    \end{equation}
\end{lem}
\begin{proof}
    Since $\pi$ is a smooth map, the exact sequence \cite[Proposition II.8.11]{MR0463157} is exact also on the left, thus 
    $$\omega_{X} = \omega_\pi\otimes\pi^*\omega_G.$$
    By Equations~\eqref{eq:omega_G} and \eqref{eq:pi_star_G} it follows:
    \begin{equation*}
        \omega_{X} =
        \omega_\pi\otimes
        (\omega_\pi\otimes\mathcal{O}_{X}(2))^{\otimes -n-2}
        =\omega_\pi^{\otimes -n-1}\otimes\mathcal{O}_{X}(-2n-4).
    \end{equation*}
    On the other hand, using Equation~\eqref{eq:det_V}
    \begin{align*}
        \mathrm{Hom}(\det(T_{X}),\det(\V)) & =  \det(T_{X})^\vee\otimes \mathcal{O}_{X}(n+d(n+1))\otimes\omega_\pi^{\otimes n(n+1)/2}\\
        & =  \omega_{\pi}^{\otimes n(n+1)/2-n-1}\otimes\mathcal{O}_{X}(n+d(n+1)-2n-4)\\
        & =  \omega_{\pi}^{\otimes n(n+1)/2-n-1}\otimes\mathcal{O}_{X}((n+1)(d-1)-3).
    \end{align*}
    as stated in Equation~\eqref{eq:Hom}.
\end{proof}
Finally, we can proceed with the proof of the proposition.
\begin{proof}[Proof of Proposition \ref{prop:conditionsnd}]
    By Equation~\eqref{eq:Hom}, $\V$ is relatively orientable if and only if $$\omega_{\pi}^{\otimes (n+1)(n/2-1)}\otimes\mathcal{O}_{X}((n+1)(d-1)-3)$$
    is the square of a line bundle. Since $\omega_\pi$ and $\mathcal{O}_{X}(1)$ generate $\mathrm{Pic}(X)$ (for example, by \eqref{eq:pi_star_G}), we need that $(n+1)(n/2-1)$ and $(n+1)(d-1)-3$ to be even.
    By Lemma~\ref{lem:cong23}, it is easy to see that it must be $d$ even and $n\equiv 2 \,\mathrm{mod} \,4$.
\end{proof}
So, we know exactly when $\V$ is relatively orientable. In particular, there is a well-defined Euler class. Now, we define explicitly a relative orientation and we will prove that it is compatible with the coordinates given in Definition~\ref{defn:my_Nis}. 
For each open set $U_{i\alpha}$, let us consider the following map $\lambda_{i\alpha}\colon\det 
{{T_X}|}_{U_{i\alpha}}\rightarrow\det
{{\V}|}_{U_{i\alpha}}$ where
$$\lambda_{i\alpha}((-1)^{\alpha}\cdot\partial_i\wedge\partial_\alpha)=\left(\frac{x_i}{x_0}\right)^{n+d(n+1)}\left(\frac{x_i}{x_0}\right)^{-n(n+1)}
\left(\frac{x_iy_\alpha}{x_0y_1-x_1y_0}\right)^{n(n+1)/2}.$$
Combining Proposition~\ref{prop:prop3.8} and \ref{prop:prop3.9}, we see that
\begin{equation}\label{eq:lambda_trans}
    \lambda_{i\alpha}((-1)^{\alpha}\cdot\partial_i\wedge\partial_\alpha)
    =\left(\frac{x_i}{x_j}\right)^{N}
    \left(\frac{x_i}{x_j}\right)^{-2M}
    \left(\frac{x_iy_\alpha}{x_jy_\beta}\right)^{M}\lambda_{j\beta}((-1)^{\alpha}\cdot\partial_i\wedge\partial_\alpha)
\end{equation}
where we denoted
$$N := (n + 1)(d - 1) - 3,  \,\,\,\,\, M:= (n + 1)(n/2 - 1).$$
It follows that there exists a well-defined morphism
$$\psi\colon\omega_{\pi}^{\otimes M}\otimes\mathcal{O}_X(N)\rightarrow\mathrm{Hom}(\det(T_{X}),\det(\V)),$$
such that
\begin{equation}\label{eq:psi_eq_lambda}
    {\psi|}_{U_{i\alpha}}\left(
    \left(\frac{x_i}{x_0}\right)^{-2M}
    \left(\frac{x_iy_\alpha}{x_0y_1-x_1y_0}\right)^{M}
    \left(\frac{x_i}{x_0}\right)^{N}\right)=\lambda_{i\alpha}.
\end{equation}
We conclude this section by proving that $\psi$ induces a trivialization compatible with the twisted coordinates, as in Definition~\eqref{defn:compatible}.
\begin{lem}\label{lem:Rel_Ori_case}
    Let $n$ and $d$ be positive integers such that $n\equiv 2 \,\mathrm{mod} \,4$, $d\equiv 0 \,\mathrm{mod} \,2$.
    The local trivializations 
    $$\left(\frac{x_i}{x_0}\right)^{n+d(n+1)}
    \left(\frac{x_i}{x_0}\right)^{-n(n+1)}
    \left(\frac{x_iy_\alpha}{x_0y_1-x_1y_0}\right)^{n(n+1)/2}$$
    of $\det \V$ are compatible with the Nisnevich coordinates $\{(U_{i\alpha},\tilde\varphi_{i\alpha})\}$ and the relative orientation $(\omega_{\pi}^{\otimes M/2}\otimes\mathcal{O}_X(N/2),\psi)$.
\end{lem}
\begin{proof}
    As we saw in Proposition~\ref{prop:conditionsnd}, $N$ and $M$ are even, so it makes sense to consider $N/2$ and $M/2$. By construction, $\lambda_{i\alpha}$ maps the distinguished basis of the anticanonical bundle to a distinguished basis of $\det\V$. Finally, Equation~\eqref{eq:psi_eq_lambda} implies that 
    $${\psi|}_{U_{i\alpha}}\left(\left(\left(\frac{x_i}{x_0}\right)^{-M}\left(\frac{x_iy_\alpha}{x_0y_1-x_1y_0}\right)^{M/2}\left(\frac{x_i}{x_0}\right)^{N/2}\right)^{\otimes 2}\right)
    =\lambda_{i\alpha}.$$
    Hence, $\lambda_{i\alpha}$ is a square.
\end{proof}

\subsection{Non-relatively orientable case}\label{subs:non_rel_ori}
Let us assume that $n\equiv 2 \,\mathrm{mod} \,4$, and that $d$ is odd. So $N$ is odd, and $\V$ cannot be relatively orientable. We show that, in this case, $\V$ is relatively orientable relative to the divisor $D=\ev^*(\{x_0=0\})$. We adopt the same approach of \cite[Lemma 3.11]{MR4211099}.

Since $\mathcal{O}_X(D)=\mathcal{O}_X(1)$ and $M$ and $N+1$ are even, by Equation~\eqref{eq:Hom} we have
$$\mathrm{Hom}(\det(T_{X}),\det(\V))\otimes\mathcal{O}_X(D) = 
\left( \omega_{\pi}^{\otimes M/2} \otimes\mathcal{O}_{X}\left(\frac{N+1}{2}\right)\right)^{\otimes 2},$$
that is, it is a tensor square. As $\mathcal{O}_X(D)$ is locally trivialized by $(\frac{x_i}{x_0})$, we may define on $U_{i\alpha}$

\begin{equation}
    {\tilde{\psi}|}_{U_{i\alpha}}
    \left(\left(\frac{x_i}{x_0}\right)^{-2M}
    \left(\frac{x_iy_\alpha}{x_0y_1-x_1y_0}\right)^{M}
    \left(\frac{x_i}{x_0}\right)^{N+1}\right)=\lambda_{i\alpha}.
\end{equation}

\begin{lem}
    Let $n$ and $d$ be positive integers such that $n\equiv 2 \,\mathrm{mod} \,4$, $d\equiv 1 \,\mathrm{mod} \,2$.
    The local trivializations 
    $$\left(\frac{x_i}{x_0}\right)^{n+d(n+1)}
    \left(\frac{x_i}{x_0}\right)^{-n(n+1)}
    \left(\frac{x_iy_\alpha}{x_0y_1-x_1y_0}\right)^{n(n+1)/2}$$
    of $\det \V$ are compatible with the Nisnevich coordinates $\{(U_{i\alpha},\tilde\varphi_{i\alpha})\}$ and the relative orientation $(\omega_{\pi}^{\otimes M/2}\otimes\mathcal{O}_X((N+1)/2),\tilde\psi)$ relative to the divisor $D$.
\end{lem}
\begin{proof}
    The canonical section $1_D$ of $\mathcal{O}_X(D)$ is locally given by (the pullback of) $\frac{x_i}{x_0}$. By construction,
    \begin{multline*}
        \lambda_{i\alpha}((-1)^{\alpha}\cdot\partial_i\wedge\partial_\alpha)\otimes \left(\frac{x_i}{x_0}\right) \\
        =\left(\frac{x_i}{x_0}\right)^{n+d(n+1)}
    \left(\frac{x_i}{x_0}\right)^{-n(n+1)+1}
    \left(\frac{x_iy_\alpha}{x_0y_1-x_1y_0}\right)^{n(n+1)/2},
    \end{multline*}
    so by Equation~\eqref{eq:lambda_trans},
    \begin{multline*}
        \lambda_{i\alpha}((-1)^{\alpha}\cdot\partial_i\wedge\partial_\alpha)
    \otimes 
    \left(\frac{x_i}{x_0}\right)\\
    =\left(\frac{x_i}{x_j}\right)^{N+1}
    \left(\frac{x_i}{x_j}\right)^{-2M}
    \left(\frac{x_iy_\alpha}{x_jy_\beta}\right)^{M}\lambda_{j\beta}((-1)^{\alpha}\cdot\partial_i\wedge\partial_\alpha)\otimes \left(\frac{x_j}{x_0}\right).
    \end{multline*}
    Thus the maps ${{\tilde\psi}|}_{U_{j\beta}}$ and ${{\tilde\psi}|}_{U_{i\alpha}}$ differ by the transition function
    $$\left( \left(\frac{x_i}{x_j}\right)^{(N+1)/2}
    \left(\frac{x_i}{x_j}\right)^{-M}
    \left(\frac{x_iy_\alpha}{x_jy_\beta}\right)^{M/2} \right)^{\otimes 2}.$$
    So $\{(U_{i\alpha},\tilde\varphi_{i\alpha})\}$ is well-defined, and the relative orientation is a square by construction. This proves the lemma.
\end{proof}

\section{Euler Class}

In this section we compute the Euler class of $\V$. Although the result is anticipated by the fact that $X$ is odd-dimensional (see Remark~\ref{rem:odd_dim_is_mH}), we adopt the classical strategy of defining a section with only one zero, compute the Scheja--Storch form, and finally apply Equation~\eqref{eq:euler}.
\begin{thm}\label{thm:4.4}
    Let $n$ and $d$ be positive integers such that the following conditions are satisfied
    \begin{itemize}
        \item $n\equiv 2 \,\mathrm{mod} \,4$,
        \item $d\equiv 0 \,\mathrm{mod} \,2$.
    \end{itemize}
    The vector bundle $\V$
    is relatively orientable. Moreover, if $d\ge n$, the Euler class of $\V$ is
    \begin{equation}\label{eq:main}
        e(\V)=d\frac{n!}{2} \mathbb{H}.
    \end{equation}
\end{thm}

\begin{proof}
    By Proposition~\ref{prop:conditionsnd}, $\V$ is relatively orientable and its Euler class is well-defined.
    Consider the global section $f$ of $\mathcal{O}_{\mathbb P^{n+1}_k}(d)$, and the global sections $\{g_j\}_{j=1}^{n}$ of $\mathcal{O}_{\mathbb P^{n+1}_k}(1)$ given by
    $$f=x_0^d+\sum_{i=1}^{n} x_{n+1}^{d-i}x_i^i, \,\,\,\,\,  g_j = x_j.$$
    Combining $f$ and $\{g_j\}_{j=1}^{n}$, we obtain a global section of $\mathcal{O}_{\mathbb P^{n+1}_k}(d)\oplus\mathcal{O}_{\mathbb P^{n+1}_k}(1)^{\oplus n}$,
    with zeros in $X$ consist only of 
    one $k$-point $p$ in $U_{n+1}:=\{x_{n+1}\neq 0\}\cong\mathbb A_k^{n+1}$.

    Let $Z:=\ev^{-1}(p)\cong \mathbb P_{k}^n$. Clearly ${{\mathcal{O}_{X}(1)}|}_{Z}$ is trivial and by Equation~\eqref{eq:omega_Z}  ${{\omega_{\pi}}|}_Z=\mathcal{O}_Z(1)$. Thus any bundle of principal parts $\P^t(d)$ splits when restricted to $\ev^{-1}(U_{n+1})=U_{n+1}\times Z$ as $\mathrm{Ext}^1(\mathcal{O}_Z,\mathcal{O}_Z(t))=0$ for $t\ge 0$. 

    Let $\partial^n f$ the induced section of $\P^n(d)$. Consider the section $\Lambda=(\partial^n f, g_1,\ldots,g_{n})$ of $\V$. Note that, by evaluation, all zeros of $\Lambda$ are contained in $\ev^{-1}(U_{n+1})$. The surjection
    \begin{equation}\label{eq:surj}
        \tau\colon\V \longrightarrow \ev^*(\mathcal{O}_{\mathbb P^{n+1}_k}(d)\oplus\mathcal{O}_{\mathbb P^{n+1}_k}(1)^{\oplus n}),
    \end{equation}
    restricted to $\ev^{-1}(U_{n+1})$ splits because all the $\P^t(d)$ do. Thus the kernel of \eqref{eq:surj} is
    \begin{equation*}
        \ker (\tau)=\mathcal{O}_{U_{n+1}}\otimes\bigoplus_{i=1}^n\mathcal{O}_Z(i).
    \end{equation*}
    
    Let $z_i:=\ev^*({{{x}_i}|}_{U_{n+1}})=\ev^*(x_i/x_{n+1})$ and consider the pullback of ${f|}_{U_{n+1}}$
    $$\ev^*({f|}_{U_{n+1}})=z_0^d+\sum_{i=1}^{n} z_i^i.$$
    Computing the Hasse derivatives of $f$, we see that ${\Lambda|}_{\ev^{-1}({U_{n+1}})}$ is the global section of $\ker(\tau)\oplus\mathcal{O}_{U_{n+1}}^{\oplus n+1}$, given by
    \begin{equation}\label{eq:lambda}
        {\Lambda|}_{\ev^{-1}({U_{n+1}})}=(f^{(1)},f^{(2)}, \ldots, f^{(n)},         \ev^*({f|}_{U_{n+1}}), z_1,z_2,\ldots,z_n),
    \end{equation}
    where for all $i=1,\ldots,n$,
    $$f^{(i)}:=\binom{d}{i}z_0^{d-i}(\d z_0)^i+(\d z_i)^i+\binom{i+1}{i}z_{i+1}(\d z_{i+1})^i+\cdots +\binom{n}{i}z_{n}^{n-i}(\d z_{n})^i,$$
    and $\{\d z_0,\ldots,\d z_n\}$ are independent global sections of $\mathcal{O}_Z(1)$. The section in Equation~\eqref{eq:lambda} has a unique (non-simple, rational) zero in $\ev^{-1}(U_{n+1})$, and it is in the open set ${U_{n+1}} \times U_0$, where $U_0:=\{\d z_0\neq 0\}$. This open set is a compatible Nisnevich neighborhood around the zero by Lemma~\ref{lem:Rel_Ori_case}. Finally, setting $w_i:=\d z_i/\d z_0$,
    \begin{equation}
        \mathcal{O}_{\{\Lambda=0\},0} \cong\frac{k[w_1,\ldots,w_n, z_0,\ldots, z_n]_0}{(f^{(1)},f^{(2)},\ldots,f^{(n)},\ev^*({f|}_{{U_{n+1}}}),{z}_1,\ldots,{z}_n)}\label{eq:ffff}
    \end{equation}
    \begin{align}
        &\cong \frac{k[w_1,\ldots,w_n, z_0,\ldots, z_n]_0}{\left(\binom{d}{1}z_0^{d-1}+w_1,\binom{d}{2}z_0^{d-2}+w_2^2,\ldots,\binom{d}{n}z_0^{d-n}+w_n^n,z_0^d,{z}_1,\ldots,{z}_n\right)} \label{eq:zwzw} \\
        &\cong \frac{k[w_1,\ldots,w_n, z_0]_0}{\left(\binom{d}{1}z_0^{d-1}+w_1,\binom{d}{2}z_0^{d-2}+w_2^2,\ldots,\binom{d}{n}z_0^{d-n}+w_n^n,z_0^d\right)}
        \otimes
        \frac{k[z_1,\ldots, z_n]_0}{({z}_1,\ldots,{z}_n)}.\label{eq:bigotimes_ring_SS}
    \end{align}
    The passage from \eqref{eq:ffff} to \eqref{eq:zwzw} is possible because the ideal $(f^{(i)}, z_1,\ldots,z_n)$ equals $$\left(\binom{d}{i}z_0^{d-i}+w_i,z_1,\ldots,z_n\right),$$as $f^{(i)}-\binom{d}{i}z_0^{d-i}-w_i\in (z_1,\ldots,z_n)$. The same holds for $\ev^*({f|}_{U_{n+1}})-z_0^d$. 
    Using functoriality, the class of the Scheja--Storch form of \eqref{eq:bigotimes_ring_SS} is the product of the classes of the forms computed in each ring. For the ring on the right it is $\langle 1 \rangle$ by \cite[Lemma~5]{MR4073493}. For the other ring, a direct computation (e.g., by induction on $n$) shows that its class is 
    $$d\frac{n!}{2}\mathbb H,$$
    so Equation~\eqref{eq:main} is proved.
\end{proof}
\begin{rem}\label{rem:odd_dim_is_mH}
    A priori, we can deduce that $e(\V)$ is a multiple of $\mathbb H$ by applying \cite[Proposition~19]{MR4237952} to $\mathcal E=\mathcal{O}_{X}(1)^{\oplus n}$ and $\mathcal E'=\Pnd$. After that, me may apply \cite[Lemma~5]{MR4237952} paired with \cite[Proposition~6.1]{MR4256011}.
\end{rem}

We are now ready to prove the main theorem.
\begin{proof}[Proof of Theorem \ref{thm:main-1}]
    Suppose that $Y$ and $\mathcal{L}$ are given by, respectively, polynomials $f$ and $\{g_j\}_{j=1}^n$. Since $\mathcal{L}$ is $k$-rational, all $g_j$ are linear. For each $p\in Y\cap \mathcal{L}$, let $U$ be an affine open set containing $p$. We can construct a morphism
    \begin{eqnarray}
    \Lambda\colon U'& \longrightarrow & \mathbb A^{2n+1}_k \label{eq:LambdaUprimo}\\
    x & \longmapsto & (\partial^n \ev^*(f)(x),\ev^*(g_1)(x),\ldots,\ev^*(g_n)(x)),\nonumber
    \end{eqnarray}
    defined in any affine open set $U'\cong U\times\mathbb A^{n}_k$ mapping surjectively to $U$.
    Consider the $k(p)$ lift $\Lambda_{k(p)}$ of $\Lambda$. Since the extension of scalars is compatible with pullback, and with the differential morphism $\partial^n$ (see \cite[(16.7.9.1)]{MR238860}), we get
    $$\Lambda_{k(p)}(x)=(\partial^n\ev^*(f_{k(p)})(x),\ev^*(g_1)(x)_{k(p)},\ldots,\ev^*(g_n)(x)_{k(p)}).$$
    We may suppose, after a translation, that a $k(p)$ lift $\tilde p$ of $p$ is the origin of $U_{k(p)}$. Moreover, after a linear transformation, we may suppose that $\ev^*(g_i)(x)=z_i$ where $\{z_0,\ldots,z_n\}$ are coordinates of $U$. Note that $f_{k(p)}$ decomposes as sum of homogeneous polynomials around $0$, thus
    $$\ev^*(f_{k(p)})=\sum_{i=1}^d\ev^*(f^{(i)}_{k(p)})=
    \sum_{i=1}^d f_{k(p)}(z_0,\ldots,z_n)^{(i)}.$$
    In the ring $R=k(p)[\d z_0,\ldots,\d z_n,z_0,\ldots,z_n]$, the $t^\mathrm{th}$ Hasse derivative $\d^t \ev^*(f_{k(p)})^{(i)}$ is contained in the ideal $(z_0,\ldots,z_n)$ if $t<i$. Moreover, since the intersection of $Y$ and $\mathcal{L}$ is transverse at $p$, there must be an $\alpha\in k(p)\setminus\{0\}$ such that $\ev^*(f_{k(p)})-\alpha z_0$ is contained in the ideal $(z_0,\ldots,z_n)$. It follows:
    $$\frac{R}{(\Lambda_{k(p)})}\cong 
    \frac{k(p)[\d z_0,\ldots,\d z_n]}{(f_{k(p)}^{(1)},\ldots,f_{k(p)}^{(n)})}\otimes
    \frac{k(p)[z_0,\ldots,z_n]}{(z_0,\ldots,z_n)}.$$
    Since $Y$ and $\mathcal{L}$ are general, the polynomials $f_{k(p)}^{(1)},\ldots,f_{k(p)}^{(n)}$ meet at finite number of reduced points. If $l$ is any of those points, by \cite[Lemma~5.5]{MR4211099} it follows
    $$\deg_{(l,\tilde p)}^{\mathbb A^1_{k(p)}}(\Lambda_{k(p)})=\deg_{l}^{\mathbb A^1_{k(p)}}(f_{k(p)}^{(1)},\ldots,f_{k(p)}^{(n)})=\mathrm{Tr}_{k(l)/k(p)}\langle J(l)\rangle.$$
    Thus the sum $\sum_l \deg_{(l,\tilde p)}^{\mathbb A^1_{k(p)}}(\Lambda_{k(p)})$ is $J(Y,p)$.
    The polynomials $f$ and $\{g_j\}_{j=1}^n$ define a global section $\Psi$ of $\V$. Hence by Theorem~\ref{thm:4.4},
    \begin{equation}\label{eq:0}
        \sum_{x:\Psi=0}\deg_x^{\mathbb A^1_{k}}(\Psi)=d\frac{n!}{2}\mathbb H.
    \end{equation}
    On the other hand, for each $x$ we may find a sistem of Nisnevich coordinates $U'$ such that ${\Psi|}_{U'}$ is represented like in Equation~\eqref{eq:LambdaUprimo}. Hence
    \begin{align}
        \sum_{x:\Psi=0}\deg_x^{\mathbb A^1_{k}}(\Psi) &=
        \sum_{p\in Y\cap\mathcal{L}} \sum_{l:\ev(l)=p} \mathrm{Tr}_{k(l)/k} \langle \deg_{(l,\tilde p)}^{\mathbb A^1_{k(l)}}\Lambda_{k(l)}\rangle \label{eq:1}\\
         &= \sum_{p\in Y\cap\mathcal{L}}\mathrm{Tr}_{k(p)/k} \left(
         \sum_{l:\ev(l)=p} \mathrm{Tr}_{k(l)/k(p)}\langle \deg_{(l,\tilde p)}^{\mathbb A^1_{k(l)}} \Lambda_{k(l)}\rangle\right) \label{eq:2}\\
          &= \sum_{p\in Y\cap\mathcal{L}}\mathrm{Tr}_{k(p)/k} \left(\sum_{l:\ev(l)=p}\deg_{(l,\tilde p)}^{\mathbb A^1_{k(p)}} \Lambda_{k(p)}
         \right) \label{eq:3}\\
         &= \sum_{p\in Y\cap\mathcal{L}}\mathrm{Tr}_{k(p)/k} \langle J(Y,P)\rangle.\label{eq:4}
    \end{align}
     We used \cite[Corollary~1.4]{MR4162156} in Equations~\eqref{eq:1} and \eqref{eq:3}. In Equation~\eqref{eq:2}, we used the decomposition of the trace of a tower of fields, see \cite[Theorem~3.8.5.(3)]{MR2459247}.
     Finally, combining Equations~\eqref{eq:0} and \eqref{eq:4}, we conclude the proof.
\end{proof}
\subsection{Final remarks}
\begin{rem}\label{rem:McKean}
    In the proof, we used the geometric interpretation of $J(l)$ given in Section~5 of McKean's article. There is another possible way to prove our result, which uses the main result of his paper. We give a sketch of it here.
    One may observe that, by \cite[Theorem~1.2]{MR4211099}, $$J(Y,p)=\frac{n!}{2}\mathbb H\in\GW(k(p)).$$
    Together with $\mathrm{Tr}_{k(p)/k}(\mathbb H) =[k(p):k]\mathbb H$ (see \cite[Lemma~2.12]{kim2023global}), and
    \begin{align*}
    \sum_{p\in Y\cap\mathcal{L}} [k(p):k]&= d,
    \end{align*}
    these would imply Theorem~\ref{thm:main-1}.
\end{rem}

\begin{rem}
    One may ask to compute the lines meeting a smooth hypersurface $Y\subset \mathbb P^{n+1}_k$ with contact order $2n+1$. For example,  inflectional tangent lines to a plane curve. These would be given by a zero section of $\P^{2n}(d)$, but that bundle is not relatively orientable for any value of $n$. 
\end{rem}
One may extend the result when $d$ is odd and $k=\mathbb R$. Following Subsection~\ref{subs:non_rel_ori}, we consider a divisor $D=\ev^*(H)$ where $H\subset \mathbb P^{n+1}_k$ is a hyperplane. Thus $\V$ is relatively orientable with respect to $D$, so Lemma~\ref{lem:LV21} applies to $E=\V$.

In general, the space $H^0(E)^\circ\setminus V_D$  could be totally disconnected, so there might not be interesting enumerative results.

However, we believe that it happens something similar to \cite[Lemma~4.1]{MR4253146}. That is, $H^0(\V)^\circ\setminus V_D$ has only two connected components. For any section
\begin{equation}\label{eq:sections}
    (f,g_1,\ldots,g_n)\in H^0(\mathcal{O}_{\mathbb P^{n+1}_k}(d)\oplus\mathcal{O}_{\mathbb P^{n+1}_k}(1)^{\oplus n}),
\end{equation}
let us denote by $\mathcal{L}$ the locus $\{g_1=\ldots=g_n=0\}$, and by
$$A\subset H^0(\mathcal{O}_{\mathbb P^{n+1}_k}(d)\oplus\mathcal{O}_{\mathbb P^{n+1}_k}(1)^{\oplus n})$$ 
the space of all sections such that either
\begin{enumerate}
    \item $\dim(\mathcal{L}\cap H)\ge 1$, or
    \item $p=\mathcal{L}\cap H$ is a $\mathbb R$-point and $f(p)=0$.
\end{enumerate}
One can easily show that only the second condition makes $A$ a subspace of codimension $1$. Thus, the complement of $A$ has two connected components, where one is the locus of sections as in Equation~\eqref{eq:sections} such that $\mathcal{L}$ is a line not contained in $H$ and $f(p)>0$. The other component is obtained by the involution
\begin{equation}\label{eq:involution}
    (f,g_1,\ldots,g_n)\longmapsto (-f,g_1,\ldots,g_n).
\end{equation}
By the continuous map
$$H^0(\V)\setminus V_D\longrightarrow
H^0(\mathcal{O}_{X}(d)\oplus\mathcal{O}_{X}(1)^{\oplus n})\setminus \ev^*(A),$$
We deduce that $H^0(\V)^\circ\setminus V_D$ has at least two connected components. We believe that there are only two, and they are isomorphic through an involution similar to \eqref{eq:involution}. With minor changes, the proof of Theorem~\ref{thm:4.4} is still valid in this case.


In the appendix of the arXiv version of this paper, we exhibit a computation described by the theorem for two cubic surfaces, $Y_+$ and $Y_-$, meeting a line $\mathcal{L}$, with respect to a relative divisor. Our example is set for $k=\mathbb Q$, but can easily be adapted to $\mathbb R$. Thus, Lemma~\ref{lem:LV21} implies that, in the connected components of sections defined by $(Y_+,\mathcal{L})$ and $(Y_-,\mathcal{L})$, the theorem holds true. In order to perform the cohomological computations, we used \textit{Macaulay}2 and \verb|OSCAR| (see \cite{M2,OSCAR-book}).

Finally, we may extend the theorem even to other values of $n$. For example, when $n=1$ osculating lines are exactly tangent lines to a point of a curve $Y$. We plan to address this problem in the future.

\appendix
\section{An explicit example}
Let $k=\mathbb Q$, and let $Y_\pm\subset\mathbb{P}^3_k$ be two cubics given by the polynomials 
\begin{equation}\label{eq:hom_Y}
    f_\pm = x_0^3+(x_2^2 \pm x_3^2)x_0-x_1^3+x_2^3+ x_3^3.
\end{equation}
The intersection of $Y_\pm$ with the line $\mathcal{L}:x_2=x_3=0$ is the subscheme consisting of the union of a $\mathbb Q$-point $p$, and a $\mathbb Q(\e)$-point $q$, where $\e$ is a root of $x^2+x+1$. That is
\begin{equation*}
    \mathrm{Proj}\frac{\mathbb{Q}[x_0,x_1,x_2,x_3]}{(f_\pm,x_2,x_3)}=\mathrm{Proj}\frac{\mathbb{Q}[x_0,x_1]}{(x_0-x_1)}\bigoplus\mathrm{Proj}\frac{\mathbb{Q}[x_0,x_1]}{(x_0^2+x_0x_1+x_1^2)}=p\cup q.
\end{equation*}
By Serre's Tor formula (\cite{MR0201468} or \cite[Appendix A]{MR0463157}), a line $l$ is osculating at $x\in \mathbb P^3_k$ if and if
$$\mathrm{length}(\mathcal{O}_{Y_\pm,x}\otimes_{\mathcal{O}_{\mathbb P^3_k,x}} \mathcal{O}_{l,x})\ge 3.$$
We determinate explicitly the ideal sheaves of the osculating lines at the points $p$, and $q$. After that, we will see the computation as in Theorem~\ref{thm:main-1}, but in the relatively oriented case.

\subsection{The rational point}
Let us compute the osculating lines at $p$. Let us see the case $Y_-$ first. It is convenient to see $f_-$ in the following way:
\begin{equation}
    (x_0-x_1)(x_0^2+x_0x_1+x_1^2)+(x_2-x_3)(x_2+x_3)x_0+x_2^3+x_3^3=0.
\end{equation}
It is clear that the tangent plane at $p$ is given by the equation $x_0-x_1$. Note that the residual part of degree $2$ decomposes as the union of two lines: $x_2-x_3$ and $x_2+x_3$.
Thus we have the ideals $I=(x_0-x_1,x_2-x_3)$ and $J=(x_0-x_1,x_2+x_3)$. The multiplicity of intersection at $p$ between $I$ and $Y_-$ is $3$. The cohomological computations in this section are available\footnote{See \url{https://github.com/mgemath/OsculatingLines_computations}.}.
We do not need to confirm the multiplicity also for the second line, as that is contained in $Y_-$.

Now, let us see the case $Y_+$.  It is convenient to see $f_+$ is the following way:
\begin{equation}
    (x_0-x_1)(x_0^2+x_0x_1+x_1^2)+(x_2^2+x_3^2)x_0+x_2^3+x_3^3=0.
\end{equation}
Arguing like before, note that we have only one subvariety $Z$ whose multiplicity $i_p(Z,Y_+)\ge3$, and it is given by the ideal  $I=(x_0-x_1,x_2^2+x_3^2)$. In particular, we do not have rational osculating lines, but an irreducible degree $2$ subvariety that decomposes as the union of two lines over $\mathbb{Q}(i)$. We expect that $i_p(Z,Y_+)$ is the sum of the multiplicities of those two complex lines.
\subsection{The complex point}
Unsurprisingly, the analysis does not change if we consider the other point. That is, when the rational "osculating subvarieties" are two, each of them has multiplicity $3$ (if not contained). When such subvarieties are unique, the multiplicity is $6$.

In the case $Y_-$, the two osculating subvarieties are defined by the ideals, respectively, $I=(x_0^2+x_0x_1+x_1^2,x_2-x_3)$ and $J=(x_0^2+x_0x_1+x_1^2,x_2+x_3)$. The second is contained in $Y_-$. 

In the case $Y_+$, the unique osculating subvariety is the scheme whose defining ideal is $I=(x_0^2+x_0x_1+x_1^2,x_2^2+x_3^2)$. Its multiplicity is that we expect.
\subsection{The enriched count}
Using enriched enumerative geometry, we need to count the osculating lines of $Y_\pm$, whose tangent point lies on the line $x_2=x_3=0$, in such a way that the count is $3\mathbb{H}\in\GW(\mathbb{Q})$.
Consider the affine subset of the Grassmannian $\mathbb{A}^4\subset G$, parameterizing lines in $\mathbb{P}^3$ not meeting the line $x_0=x_2=0$. For each point $(a,b,c,d)\in\mathbb{A}^4$ we associate the line
\begin{align*}
    x_1 &=  ax_0 + bx_2\\
    x_3 &=  cx_0 + dx_2.
\end{align*}
The restriction of the universal family  $\pi\colon X\rightarrow G$ to this open subset is the projection $\pi\colon\mathbb{A}^4\times \mathbb{P}^1\rightarrow\mathbb{A}^4$. The evaluation map $\ev\colon X\rightarrow \mathbb{P}^3$ is
\begin{eqnarray*}
    \ev\colon\mathbb{A}^4\times \mathbb{P}^1 & \longrightarrow & \mathbb{P}^3\\
    (a,b,c,d)\times [\mu:\lambda] & \mapsto & [\mu:a\mu+b\lambda:\lambda:c\mu+d\lambda].
\end{eqnarray*}
We can take an open subset $\mathbb{A}^5\subset X$ by taking $\{\mu\neq0\}$. This is equivalent to considering the problem outside the divisor $x_0=0$. The pullback of $f_\pm$ under $\ev$ defines the following affine equation in $\mathbb{A}^5$
\begin{multline*}
f^\circ_\pm(a,b,c,d,\lambda) = 1+c^3\pm c^2-a^3+(3 c^2 d-3 a^2 b\pm2 c d)\lambda\\
+(1-3 a b^2+3 c d^2\pm d^2)\lambda^2+(1-b^3+d^3)\lambda^3.
\end{multline*}

In the same way, $x_2=x_3=0$ define $\lambda=c+d\lambda=0$. Thus, the pointed lines we are looking for are parameterized by points $(a,b,c,d,\lambda)\in\mathbb{A}^5$ such that:
\begin{equation}\label{eq:system}
    f^\circ_\pm  \bigg|_{\lambda=0}=\frac{\partial}{\partial\lambda}f^\circ_\pm\bigg|_{\lambda=0}=\frac{1}{2}\frac{\partial^2}{\partial\lambda^2}f^\circ_\pm\bigg|_{\lambda=0}=\lambda=c+d\lambda=0.
\end{equation}
Note that we used the Hasse derivatives. Hence, we have to solve the following system of equations.
\begin{align*}
    1+c^3\pm c^2-a^3  &= 0\\
    3 c^2 d-3 a^2 b\pm2 c d &= 0\\
    1-3 a b^2+3 c d^2\pm d^2 &= 0\\
    \lambda &= 0\\
    c+d\lambda &= 0.
\end{align*}
Its solutions are the same of the following system 
\begin{align*}
    a^3  &= 1\\
    b=c=\lambda &= 0\\
    \pm d^2 &= -1.
\end{align*}
The Jacobian of \eqref{eq:system} is
\begin{multline*}
    \mathcal{J}_\pm=\pm 18a^4d+54a^4cd-27a^4d^2\lambda+27b^2 (\mp2 a^2 c \lambda - 3 a^2 c^2 \lambda) \\
    +18b (\mp2 a^3 c - 3 a^3 c^2 \pm 2 a^3 d \lambda + 
    6 a^3 c d \lambda).
\end{multline*}
Fortunately, when we evaluate $\mathcal{J}_\pm$ at any solution, the result is simply
\begin{equation}
    \mathcal{J}_\pm=\pm 18 a^4 d=\pm 3^22ad.
\end{equation}
Let us compute $\sum_l\mathrm{Tr}_{k(l)/k}(\mathcal J_\pm|_{l})$ where $l$ runs among the solutions. 

Let us consider the case $Y_+$. The first solution is the $k(l)$-point where
\begin{equation*}
    \frac{\mathbb{Q}[a,b,c,d,\lambda]}{(a-1,b,c,d^2+1,\lambda)}\cong \frac{\mathbb{Q}[d]}{(d^2+1)}=: k(l).
\end{equation*}
There is a symmetric $k(l)$-bilinear form defined by $\langle\mathcal J_+\rangle=\langle 2d\rangle$, that is,
\begin{eqnarray*}
    \mathrm{Sym}^2k(l) & \longrightarrow & k(l)\\
    z\otimes w & \longmapsto & 2dzw.
\end{eqnarray*}
If we compose it with the trace map\footnote{The trace map of $\mathbb Q(i)/\mathbb Q$ sends a complex number $z$ to $\Re (2z)$.} $k(l)\rightarrow k$, we get a symmetric $\mathbb Q$-bilinear form of rank two with eigenvalues $-4$ and $4$ and eigenvectors $\{\frac{\sqrt{2}}{2}+ d\frac{\sqrt{2}}{2},\frac{\sqrt{2}}{2}-d\frac{\sqrt{2}}{2}\}$.
Its class is $\langle4\rangle+\langle-4\rangle=\mathbb H\in \GW(\mathbb Q)$. The other solution is the $k(l)$-point where
\begin{equation*}
    \frac{\mathbb{Q}[a,b,c,d,\lambda]}{(a^2+a+1,b,c,d^2+1,\lambda)}\cong \frac{\mathbb{Q}[a,d]}{(a^2+a+1,d^2+1)}=: k(l).
\end{equation*}
The $k(l)$-bilinear form defined by $\langle\mathcal J_+\rangle=\langle 2ad\rangle$, composed with the trace map
\begin{equation*}
\begin{array}{cccc}
 \mathrm{Tr}_{k(l)/k}\colon &  k(l) & \longrightarrow & k \\
  & \alpha+\beta a + \gamma d +\delta ad & \longmapsto & 4\alpha-4\beta,
\end{array}
\end{equation*}
defines a symmetric $\mathbb Q$-bilinear form. Using the eigenvectors
$$
\left\{\frac{\sqrt{2}}{2}+ d\frac{\sqrt{2}}{2},\frac{\sqrt{2}}{2}-d\frac{\sqrt{2}}{2},
a\frac{\sqrt{2}}{2}+ ad\frac{\sqrt{2}}{2},a\frac{\sqrt{2}}{2}-ad\frac{\sqrt{2}}{2}\right\},
$$
we see that the eigenvalues are $8$ and $-8$, both with multiplicity two. Hence, the class of the bilinear form is $2 \mathbb H\in \GW(\mathbb Q)$. Finally,
\begin{equation*}
    \sum_l\mathrm{Tr}_{k(l)/k}(\mathcal J_+|_{l})=\mathbb H+2\mathbb H=3\mathbb H.
\end{equation*}
The case $Y_-$ is very similar, the final sum in this case is:
\begin{equation*}
    \sum_l\mathrm{Tr}_{k(l)/k}(\mathcal J_-|_{l})=\langle4\rangle+\langle-4\rangle+\mathbb H+\mathbb H=3\mathbb H.
\end{equation*}
These results agree with our main theorem. 

\bibliographystyle{amsalpha}
\bibliography{refs}

\newcommand{\etalchar}[1]{$^{#1}$}
\providecommand{\bysame}{\leavevmode\hbox to3em{\hrulefill}\thinspace}
\providecommand{\MR}{\relax\ifhmode\unskip\space\fi MR }
\providecommand{\MRhref}[2]{%
  \href{http://www.ams.org/mathscinet-getitem?mr=#1}{#2}
}
\providecommand{\href}[2]{#2}
\begin{thebibliography}{DGGM23}

\bibitem[BBM{\etalchar{+}}21]{MR4162156}
Thomas Brazelton, Robert Burklund, Stephen McKean, Michael Montoro, and Morgan
  Opie, \emph{The trace of the local {$\mathbb A^1$}-degree}, Homology Homotopy
  Appl. \textbf{23} (2021), no.~1, 243--255. \MR{4162156}

\bibitem[BLR90]{MR1045822}
Siegfried Bosch, Werner L\"{u}tkebohmert, and Michel Raynaud, \emph{N\'{e}ron
  models}, Ergebnisse der Mathematik und ihrer Grenzgebiete (3) [Results in
  Mathematics and Related Areas (3)], vol.~21, Springer-Verlag, Berlin, 1990.
  \MR{1045822}

\bibitem[BMP23]{MR4648854}
Thomas Brazelton, Stephen McKean, and Sabrina Pauli, \emph{B\'ezoutians and the
  {$\Bbb A^1$}-degree}, Algebra Number Theory \textbf{17} (2023), no.~11,
  1985--2012. \MR{4648854}

\bibitem[BW23a]{MR4557905}
Tom Bachmann and Kirsten Wickelgren, \emph{Euler classes: six-functors
  formalism, dualities, integrality and linear subspaces of complete
  intersections}, J. Inst. Math. Jussieu \textbf{22} (2023), no.~2, 681--746.
  \MR{4557905}

\bibitem[BW23b]{MR4635342}
\bysame, \emph{On quadratically enriched excess and residual intersections}, J.
  Reine Angew. Math. \textbf{802} (2023), 77--123. \MR{4635342}

\bibitem[Caz12]{MR3059240}
Christophe Cazanave, \emph{Algebraic homotopy classes of rational functions},
  Ann. Sci. \'{E}c. Norm. Sup\'{e}r. (4) \textbf{45} (2012), no.~4, 511--534.
  \MR{3059240}

\bibitem[CDH23]{MR4626882}
Ethan Cotterill, Ignacio Darago, and Changho Han, \emph{Arithmetic inflection
  formulae for linear series on hyperelliptic curves}, Math. Nachr.
  \textbf{296} (2023), no.~8, 3272--3300. \MR{4626882}

\bibitem[CM18]{MR3773793}
K.~Cieliebak and K.~Mohnke, \emph{Punctured holomorphic curves and {L}agrangian
  embeddings}, Invent. Math. \textbf{212} (2018), no.~1, 213--295. \MR{3773793}

\bibitem[Dar80]{BSMA_1880_2_4_1_348_1}
Gaston Darboux, \emph{Sur le contact des courbes et des surfaces}, Bulletin des
  Sciences Math\'ematiques et Astronomiques \textbf{2e s{\'e}rie, 4} (1880),
  no.~1, 348--384 (fr).

\bibitem[DEF{\etalchar{+}}24]{OSCAR-book}
Wolfram Decker, Christian Eder, Claus Fieker, Max Horn, and Michael Joswig
  (eds.), \emph{The {C}omputer {A}lgebra {S}ystem {OSCAR}: {A}lgorithms and
  {E}xamples}, 1 ed., Algorithms and {C}omputation in {M}athematics, vol.~32,
  Springer, 8 2024.

\bibitem[DGGM23]{MR4589636}
Cameron Darwin, Aygul Galimova, Miao Gu, and Stephen McKean, \emph{Conics
  meeting eight lines over perfect fields}, J. Algebra \textbf{631} (2023),
  24--45. \MR{4589636}

\bibitem[EH16]{MR3617981}
David Eisenbud and Joe Harris, \emph{3264 and all that---a second course in
  algebraic geometry}, Cambridge University Press, Cambridge, 2016.
  \MR{3617981}

\bibitem[Eis78]{MR0494226}
David Eisenbud, \emph{An algebraic approach to the topological degree of a
  smooth map}, Bull. Amer. Math. Soc. \textbf{84} (1978), no.~5, 751--764.
  \MR{494226}

\bibitem[Eis95]{MR1322960}
\bysame, \emph{Commutative algebra}, Graduate Texts in Mathematics, vol. 150,
  Springer-Verlag, New York, 1995, With a view toward algebraic geometry.
  \MR{1322960}

\bibitem[EL77]{MR0467800}
David Eisenbud and Harold~I. Levine, \emph{An algebraic formula for the degree
  of a {$C\sp{\infty }$} map germ}, Ann. of Math. (2) \textbf{106} (1977),
  no.~1, 19--44. \MR{467800}

\bibitem[Gro67]{MR238860}
A.~Grothendieck, \emph{\'{E}l\'{e}ments de g\'{e}om\'{e}trie alg\'{e}brique.
  {IV}. \'{E}tude locale des sch\'{e}mas et des morphismes de sch\'{e}mas
  {IV}}, Inst. Hautes \'{E}tudes Sci. Publ. Math. \textbf{24} (1967), no.~32,
  361. \MR{238860}

\bibitem[GS]{M2}
Daniel~R. Grayson and Michael~E. Stillman, \emph{Macaulay2, a software system
  for research in algebraic geometry}, Available at
  \url{http://www.math.uiuc.edu/Macaulay2/}.

\bibitem[Har77]{MR0463157}
Robin Hartshorne, \emph{Algebraic {G}eometry}, Springer-Verlag, New
  York-Heidelberg, 1977, Graduate Texts in Mathematics, No. 52. \MR{0463157}

\bibitem[Him77]{MR0458467}
G.~N. Him\v{s}ia\v{s}vili, \emph{The local degree of a smooth mapping},
  Sakharth. SSR Mecn. Akad. Moambe \textbf{85} (1977), no.~2, 309--312.
  \MR{458467}

\bibitem[Huy05]{MR2093043}
Daniel Huybrechts, \emph{Complex geometry}, Universitext, Springer-Verlag,
  Berlin, 2005, An introduction. \MR{2093043}

\bibitem[KP23]{kim2023global}
Hyun~Jong Kim and Sun~Woo Park, \emph{Global $\mathbb{A}^1$ degrees of covering
  maps between modular curves}, 2023.

\bibitem[KW19]{MR3909901}
Jesse~Leo Kass and Kirsten Wickelgren, \emph{The class of
  {E}isenbud-{K}himshiashvili-{L}evine is the local {$\mathbf{A}^1$}-{B}rouwer
  degree}, Duke Math. J. \textbf{168} (2019), no.~3, 429--469. \MR{3909901}

\bibitem[KW20]{MR4073493}
\bysame, \emph{A classical proof that the algebraic homotopy class of a
  rational function is the residue pairing}, Linear Algebra Appl. \textbf{595}
  (2020), 157--181. \MR{4073493}

\bibitem[KW21]{MR4247570}
\bysame, \emph{An arithmetic count of the lines on a smooth cubic surface},
  Compos. Math. \textbf{157} (2021), no.~4, 677--709. \MR{4247570}

\bibitem[Lee13]{MR2954043}
John~M. Lee, \emph{Introduction to smooth manifolds}, second ed., Graduate
  Texts in Mathematics, vol. 218, Springer, New York, 2013. \MR{2954043}

\bibitem[LP18]{MR3877435}
Maycol~Falla Luza and Jorge~Vit\'{o}rio Pereira, \emph{Extactic divisors for
  webs and lines on projective surfaces}, Michigan Math. J. \textbf{67} (2018),
  no.~4, 743--756. \MR{3877435}

\bibitem[LV21]{MR4253146}
Hannah Larson and Isabel Vogt, \emph{An enriched count of the bitangents to a
  smooth plane quartic curve}, Res. Math. Sci. \textbf{8} (2021), no.~2, Paper
  No. 26, 21. \MR{4253146}

\bibitem[McK21]{MR4211099}
Stephen McKean, \emph{An arithmetic enrichment of {B}\'{e}zout's {T}heorem},
  Math. Ann. \textbf{379} (2021), no.~1-2, 633--660. \MR{4211099}

\bibitem[McK22]{circles}
Stephen McKean, \emph{{Circles of Apollonius two ways}}, 2022,
  arXiv:2210.13288.

\bibitem[Mor06]{MR2275634}
Fabien Morel, \emph{{$\mathbb A^1$}-algebraic topology}, International
  {C}ongress of {M}athematicians. {V}ol. {II}, Eur. Math. Soc., Z\"{u}rich,
  2006, pp.~1035--1059. \MR{2275634}

\bibitem[Mor12]{MR2934577}
\bysame, \emph{{$\mathbb A^1$}-algebraic topology over a field}, Lecture Notes
  in Mathematics, vol. 2052, Springer, Heidelberg, 2012. \MR{2934577}

\bibitem[MS21]{MR4332489}
Dusa McDuff and Kyler Siegel, \emph{Counting curves with local tangency
  constraints}, J. Topol. \textbf{14} (2021), no.~4, 1176--1242. \MR{4332489}

\bibitem[MS23]{mikhalkin2023ellipsoidal}
Grigory Mikhalkin and Kyler Siegel, \emph{Ellipsoidal superpotentials and
  stationary descendants}, 2023.

\bibitem[Muk03]{MR2004218}
Shigeru Mukai, \emph{An introduction to invariants and moduli}, japanese ed.,
  Cambridge Studies in Advanced Mathematics, vol.~81, Cambridge University
  Press, Cambridge, 2003. \MR{2004218}

\bibitem[Mur21]{MR4256011}
Giosu{\`e} Muratore, \emph{A recursive formula for osculating curves}, Ark.
  Mat. \textbf{59} (2021), no.~1, 195--211. \MR{4256011}

\bibitem[MV99]{MR1813224}
Fabien Morel and Vladimir Voevodsky, \emph{{${\bf A}^1$}-homotopy theory of
  schemes}, Inst. Hautes \'{E}tudes Sci. Publ. Math. (1999), no.~90, 45--143.
  \MR{1813224}

\bibitem[MVW06]{MR2242284}
Carlo Mazza, Vladimir Voevodsky, and Charles Weibel, \emph{Lecture notes on
  motivic cohomology}, Clay Mathematics Monographs, vol.~2, American
  Mathematical Society, Providence, RI; Clay Mathematics Institute, Cambridge,
  MA, 2006. \MR{2242284}

\bibitem[Pau22]{MR4437612}
Sabrina Pauli, \emph{Quadratic types and the dynamic {E}uler number of lines on
  a quintic threefold}, Adv. Math. \textbf{405} (2022), Paper No. 108508, 37.
  \MR{4437612}

\bibitem[Sal65]{MR0200123}
George Salmon, \emph{A treatise on the analytic geometry of three dimensions.
  {V}ol. {II}}, Fifth edition. Edited by Reginald A. P. Rogers, Chelsea
  Publishing Co., New York, 1965. \MR{0200123}

\bibitem[Ser65]{MR0201468}
Jean-Pierre Serre, \emph{Alg\`ebre locale. {M}ultiplicit\'{e}s}, Lecture Notes
  in Mathematics, vol.~11, Springer-Verlag, Berlin-New York, 1965, Cours au
  Coll\`ege de France, 1957--1958, r\'{e}dig\'{e} par Pierre Gabriel, Seconde
  \'{e}dition, 1965. \MR{0201468}

\bibitem[SS75]{MR0393056}
G\"{u}nter Scheja and Uwe Storch, \emph{\"{U}ber {S}purfunktionen bei
  vollst\"{a}ndigen {D}urchschnitten}, J. Reine Angew. Math. \textbf{278/279}
  (1975), 174--190. \MR{393056}

\bibitem[SW21]{MR4237952}
Padmavathi Srinivasan and Kirsten Wickelgren, \emph{An arithmetic count of the
  lines meeting four lines in {${\bf P}^3$}}, Trans. Amer. Math. Soc.
  \textbf{374} (2021), no.~5, 3427--3451, With an appendix by Borys Kadets,
  Srinivasan, Ashvin A. Swaminathan, Libby Taylor and Dennis Tseng.
  \MR{4237952}

\bibitem[Vai81]{MR0626480}
Israel Vainsencher, \emph{Counting divisors with prescribed singularities},
  Trans. Amer. Math. Soc. \textbf{267} (1981), no.~2, 399--422. \MR{626480}

\bibitem[Voj07]{MR2349665}
Paul Vojta, \emph{Jets via {H}asse-{S}chmidt derivations}, Diophantine
  geometry, CRM Series, vol.~4, Ed. Norm., Pisa, 2007, pp.~335--361.
  \MR{2349665}

\bibitem[Wei09]{MR2459247}
Steven~H. Weintraub, \emph{Galois theory}, second ed., Universitext, Springer,
  New York, 2009. \MR{2459247}

\end{thebibliography}

\end{document}